\documentclass{amsart}

\usepackage[pdfauthor={Dimitrios Chatzakos, Yiannis N. Petridis},
pdftitle={},
            pdfkeywords={Hyperbolic lattice points},
             pdfcreator={Pdflatex}]
{hyperref}
\hypersetup{colorlinks=true}
\usepackage{xcolor}
\usepackage[normalem]{ulem}
\usepackage{marginnote}
\usepackage{enumerate}

\usepackage{amssymb,amsmath,latexsym,amsthm}
\usepackage[english]{babel}

\newtheorem{theorem}{Theorem}[section]
\newtheorem{lemma}[theorem]{Lemma}
\newtheorem{proposition}[theorem]{Proposition}
\newtheorem{corollary}[theorem]{Corollary}
\theoremstyle{definition}
\newtheorem{definition}[theorem]{Definition}

\numberwithin{equation}{section}

\begin{document}

\newtheorem{remark}[theorem]{Remark}



\def \g {{\gamma}}
\def \G {{\Gamma}}
\def \l {{\lambda}}
\def \a {{\alpha}}
\def \b {{\beta}}
\def \f {{\phi}}
\def \r {{\rho}}
\def \R {{\mathbb R}}
\def \H {{\mathbb H}}
\def \N {{\mathbb N}}
\def \C {{\mathbb C}}
\def \Z {{\mathbb Z}}
\def \F {{\Phi}}
\def \Q {{\mathbb Q}}
\def \e {{\epsilon }}
\def \ev {{\vec\epsilon}}
\def \ov {{\vec{0}}}
\def \GinfmodG {{\Gamma_{\!\!\infty}\!\!\setminus\!\Gamma}}
\def \GmodH {{\Gamma\backslash\H}}
\def \sl  {{\hbox{SL}_2( {\mathbb R})} }
\def \psl  {{\hbox{PSL}_2( {\mathbb R})} }
\def \slz  {{\hbox{SL}_2( {\mathbb Z})} }
\def \pslz  {{\hbox{PSL}_2( {\mathbb Z})} }
\def \L  {{\hbox{L}^2}}

\newcommand{\norm}[1]{\left\lVert #1 \right\rVert}
\newcommand{\abs}[1]{\left\lvert #1 \right\rvert}
\newcommand{\modsym}[2]{\left \langle #1,#2 \right\rangle}
\newcommand{\inprod}[2]{\left \langle #1,#2 \right\rangle}
\newcommand{\Nz}[1]{\left\lVert #1 \right\rVert_z}
\newcommand{\tr}[1]{\operatorname{tr}\left( #1 \right)}

\title[Hyperbolic lattice-point counting]{$\Omega$-results for the hyperbolic lattice point problem}
\author{Dimitrios Chatzakos}
\address{Department of Mathematics, University College London, Gower Street, London WC1E 6BT}
\email{d.chatzakos.12@ucl.ac.uk}
\thanks{The author was supported by a DTA from EPSRC during his PhD studies at UCL}
\date{\today}

\begin{abstract}
For $\G$ a cocompact or cofinite Fuchsian group, we study the lattice point problem on the Riemann surface $\GmodH$. The main asymptotic for the counting of the orbit $\G z$ inside a circle of radius $r$ centered at $z$ grows like $c e^r$. Phillips and Rudnick studied $\Omega$-results for the error term and mean results in $r$ for the normalized error term. We investigate the normalized error term in the natural parameter $X=2 \cosh r$ and prove $\Omega_{\pm}$-results for the orbit $\G w$ and circle centered at $z$, even for $z \neq w$.
\end{abstract}
\maketitle
\section{Introduction}\label{Introduction}

Let $\mathbb{H}$ be the hyperbolic plane and $z$, $w$ two fixed points in $\mathbb{H}$. We denote by $\rho(z,w)$ their hyperbolic distance. For $\Gamma$ a cocompact or cofinite Fuchsian group, we are interested in the problem of estimating, as $r \to \infty$, the quantity
\begin{displaymath}
N_r(z,w)  = \# \{ \gamma \in \Gamma : \rho( z, \gamma w) \leq r \}.
\end{displaymath}
The study of the asymptotic behaviour of $N_r(z,w)$ is traditionally called the hyperbolic lattice point problem. This problem has been studied by many authors, see \cite{delsarte, good, gunther, huber1, patterson, selberg}. For notational reasons let $u(z,w)$ denote the point pair-invariant function
\begin{displaymath}
u(z,w) = \frac{|z-w|^2}{4 \Im(z) \Im(w)}.
\end{displaymath}
Then $\cosh \rho(z,w) = 2 u(z,w) +1$ and, after the change of variable $X=2\cosh r$, the problem is equivalent to studying the quantity
\begin{displaymath}
N(X;z, w)  = \# \{ \gamma \in \Gamma : 4u( z, \gamma w) + 2 \leq X \}.
\end{displaymath}
 as $X \to \infty$.

Let $\Delta$ be the Laplacian of the hyperbolic surface $\GmodH$ and let $ \{\lambda_j \}_{j=0}^{\infty}$ be the discrete spectrum of $-\Delta$. Write $\lambda_j= s_j(1-s_j) =1/4 +t_j^2$, and let $ u_j $ be the $\L$-normalized eigenfunction (Maass form) with eigenvalue $\lambda_j$. 
We have the following theorem.
\begin{theorem} [Selberg \cite{selberg}, G\"unther \cite{gunther}, Good \cite{good}] \label{mainformulaclassical} Let $\Gamma$ be a cocompact or cofinite Fuchsian group. Then:
\begin{eqnarray*} 
N(X;z,w)  =  \sum_{1/2 < s_j \leq 1} \sqrt{\pi} \frac{\Gamma(s_j - 1/2)}{\Gamma(s_j + 1)} u_j(z) \overline{u_j(w)} X^{s_j}  + E(X;z,w),
\end{eqnarray*}
with
 \begin{displaymath}
 E(X;z,w) = O(X^{2/3}).
\end{displaymath}
\end{theorem}
 We are interested on the growth of the error term $E(X;z,w)$. Conjecturally
\begin{equation} \label{conjecture1}
E(X; z,w) = O_{\epsilon}(X^{1/2 + \epsilon})
\end{equation}
for every $\epsilon > 0$. 

Let $h(t) = h_X(t)$ be the Selberg/Harish-Chandra transform of the characteristic function $\chi_{[0, (X-2)/4]}$. Then, the error term has a {\lq spectral expansion\rq}, which for compact $\GmodH$ takes the form 
\begin{equation} \label{almostspectralexpansionerror}
E(X;z,w) =   \sum_{ t_j \in \R}  h(t_j) u_j(z) \overline{u_j(w)} + O \left( X^{-1} + \sum_{1/2<s_j \leq 1} \frac{X^{1-s_j} }{2s_j-1} \right),
\end{equation}
(see section \ref{result1compactcase}). The behavior of the terms coming from the eigenvalue $\lambda_j =1/4$ (corresponding to $t_j=0$) is well understood (see  \cite[p.~86, Lemma~2.2]{phirud}). It contributes
\begin{eqnarray*}
 h(0) \sum_{t_j=0} u_j(z) \overline{u_j(w)} = O ( X^{1/2} \log X ).
\end{eqnarray*}
 We subtract this quantity from $E(X;z,w)$ and we define the error term $e(X;z,w)$ to be the difference
\begin{equation} \label{seconderrorterm}
e(X;z,w) = E(X;z,w) -  h(0) \sum_{t_j=0} u_j(z) \overline{u_j(w)}.
\end{equation}
Thus, conjecture (\ref{conjecture1}) can be restated as
\begin{equation} \label{conjecture2}
e(X; z,w) = O_{\epsilon}(X^{1/2 + \epsilon})
\end{equation}
for every $\epsilon > 0$. 
The following result of Patterson supports this conjecture.
\begin{theorem} [Patterson, \cite{patterson}, \cite{patterson2}] \label{Patterson} Let $\G$ be a cocompact or cofinite Fuchsian group. Then the error term $ e(X; z,w)$ satisfies the average bound
\begin{displaymath}
\frac{1}{X} \int_{2}^{X} e(x;z,w) d x = O(X^{1/2} ).
\end{displaymath}
\end{theorem}
Further, Phillips and Rudnick proved mean value results for the error term.
\begin{theorem} [Phillips-Rudnick, \cite{phirud}]\label{philrudn1} a) If $\G$ is cocompact, then
 \begin{equation} \label{mnrt1}
\lim_{T \to \infty} \frac{1}{T} \int_{0}^{T}  \frac{e(2 \cosh r ;z,z)}{e^{r/2}} d r =0.
\end{equation} 
\\b) If $\G$ is cofinite, then
 \begin{equation} \label{mnrt2}
\lim_{T \to \infty} \frac{1}{T} \int_{0}^{T}  \frac{e(2 \cosh r;z,z)}{e^{r/2}} dr = \sum_{\frak{a}}\left| E_{\frak{a}} (z, 1/2) \right|^2,
\end{equation}
where $E_{\frak{a}}(z,s)$ is the Eisenstein series corresponding to the cusp $\frak{a}$.
\end{theorem}
In case $b)$, the  limit of the Phillips-Rudnick normalized average error term is positive only when $E_{\mathfrak{a}}(z,1/2) \neq 0$ for at least one cusp $\frak{a}$. Such an Eisenstein series is called a null-vector for $\G$. 

In this paper, by $\Omega$-results we mean lower bounds for the $\limsup |e(X;z,w)|$. That means, if $g(X)$ is a positive function,  we write $e(X;z,w) = \Omega (g(X))$ if and only if $e(X;z,w) \neq o(g(X))$, i.e.
\begin{eqnarray*}
\lim \sup \frac{|e(X;z,w)|}{g(X)} > 0.
\end{eqnarray*}
Phillips and Rudnick proved also the following $\Omega$-results for the error term $e(X;z,z)$.
\begin{theorem} [Phillips-Rudnick, \cite{phirud}]\label{philrudn2} a) If $\G$ is cocompact or a subgroup of finite index in $\pslz$, then for all $\delta >0$, 
 \begin{displaymath}
e(X;z,z) = \Omega \left(X^{1/2} (\log \log X)^{1/4 - \delta} \right).
\end{displaymath}
b) If $\G$ is cofinite but not cocompact, and either at least one eigenvalue $\lambda_j >1/4$ or a null-vector, then, 
 \begin{displaymath}
e(X;z,z) = \Omega \left(X^{1/2} \right).
\end{displaymath}
c) In any other cofinite case, for all $\delta>0$,
 \begin{displaymath}
e(X;z,z) = \Omega \left(X^{1/2-\delta} \right).
\end{displaymath}
\end{theorem}

We distinguish the two cases of $\Omega$-results: we write $e(X;z,w) = \Omega_{+}(g(X))$ if
\begin{eqnarray*}
\lim \sup \frac{e(X;z,w)}{g(X)} > 0,
\end{eqnarray*}
 and $e(X;z,w) = \Omega_{-}(g(X))$ if
\begin{eqnarray*}
\lim \inf \frac{e(X;z,w)}{g(X)} < 0.
\end{eqnarray*}
Instead of the normalization of Theorem \ref{philrudn1}, we are interested in studying the more natural normalization 
 \begin{equation} \label{mtzdefinition}
M(X;z,w) =   \frac{1}{X} \int_{2}^{X}  \frac{e(x;z,w)}{x^{1/2}} d x 
\end{equation}
as $X \to \infty$. Theorem \ref{Patterson} and integration by parts imply that $M(X;z,w) = O(1)$. 

For Theorem \ref{philrudn2}, after choosing $z=w$, Phillips and Rudnick work with the average of the function
\begin{equation}
\sum_{0 \neq t_j \in \R}  \frac{|u_j(z)|^2}{t_j^{3/2}}e^{i t_j \log X},
\end{equation} 
which is an almost periodic function in the variable $s=\log X$. In contrast with Theorem \ref{philrudn1}, we deal with $M(X; z,w)$ and for $z \neq w$. We prove that, under specific conditions, $M(X;z,w)$ does not have a limit.
\begin{theorem}\label{result1} Let $\G$ be a cocompact or cofinite Fuchsian group and $z$ a fixed point. Then there exists a fixed $\delta =\delta_{\G,z} >0$ such that for every point $w \in B(z,\delta)$ we have:
\\
a) if $\G$ is cocompact, then
\begin{displaymath}
M(X;z,w) = \Omega_{-}(1).
\end{displaymath}
\\
b) if $\G$ is cofinite and has at least one eigenvalue $\lambda_j>1/4$ with $u_j(z) \neq 0$, then
\begin{displaymath}
M(X;z,w)  - \sum_{\mathfrak{a}} \left|E_{\mathfrak{a}}(z,1/2) \right|^2 = \Omega_{-}(1).
\end{displaymath}
Further, for both cases, there exists a constant $a>1/4$ such that if $\lambda_1 >a$ the limit of $M(X;z,w)$ as $X \to \infty$ does not exist. 
\end{theorem} 
For the exact value of the constant $a$ see section \ref{lemmasproofs}. 
In many cases, Theorem \ref{result1} implies as an immediate corollary $\Omega$-results pointwise for $e(X;z,w)$ with $w \in B(z,\delta)$.
\begin{corollary} \label{coroll}
With the notation of Theorem \ref{result1} we have:
\\a) if $\G$ is either i) cocompact, or ii) as in case $b)$ of Theorem \ref{result1} and does not have null-vectors, then
 \begin{displaymath}
e(X;z,w) = \Omega_{-} \left(X^{1/2} \right)
\end{displaymath}
for every $w \in B(z,\delta)$.
\\
b) if $\lambda_1 >a$, then
 \begin{displaymath}
e(X;z,w) = \Omega \left(X^{1/2} \right)
\end{displaymath}
for every $w \in B(z,\delta)$.
\end{corollary}

Corollary \ref{coroll} does not 
cover all cases of cofinite Fuchsian groups. However, using a more careful analysis of $e(X;z,w)$, there are some more cases of cofinite groups for which we can deduce $\Omega$-results for $e(X;z,w)$. For this purpose, we have the following definition, which is related to Weyl's law (see Theorem \ref{localweylslaw}).

 \begin{definition} \label{sufficientlymany}
Let $\G$ be a cofinite Fuchsian group. We say that $\G$ has sufficiently many cusp forms at the point $z$ if 
\begin{eqnarray*}
\sum_{|t_j| <T} |u_j(z)|^2  \gg T^2. 
\end{eqnarray*} 
\end{definition}
We prove the following result.
\begin{theorem}\label{result2} Let $\G$ be a cofinite but not cocompact Fuchsian group and $z \in \H$ fixed. Then, there exists a fixed $\delta =\delta_{\G,z} >0$ such that for every point $w \in B(z,\delta)$ we have:
\\a) if $\G$ has sufficiently many cusp forms at $z$, then
 \begin{displaymath}
e(X;z,w) = \Omega_{-} \left(X^{1/2} \right).
\end{displaymath}
\\b) if $\G$ has null vectors, then
 \begin{displaymath}
e(X;z,w) = \Omega_{+} \left(X^{1/2}\right).
 \end{displaymath}
\end{theorem}
Hence, we conclude that:
\begin{corollary} \label{lastcorollary}
If $\G$ is cofinite but not cocompact, has null vectors and sufficiently many cusp forms at $z$, then 
 \begin{displaymath}
e(X;z,w) = \Omega_{\pm} \left(X^{1/2}\right).
 \end{displaymath}
\end{corollary}

 The proofs of Theorems \ref{result1} and \ref{result2} depend crucially on specific {\lq fixed sign\rq} properties of the $\Gamma$-function. We summarize these properties in section \ref{lemmasproofs} (Lemma \ref{gammalemma}). We prove Theorem \ref{result1} in section \ref{result1compactcase} and Theorem \ref{result2} in section \ref{theclassiclaproblemerrorterm}. Our detailed analysis of both the discrete and the continuous spectrum in section \ref{theclassiclaproblemerrorterm} can be used to give a second proof of case $a)$ of Corollary \ref{coroll}.

\begin{remark}
In the case of the Euclidean circle problem, Hardy was the first who proved $\Omega$-results for the error term. In fact, in \cite{hardy1} he proved that the error term is $\Omega_{\pm}(X^{1/4})$ and in \cite{hardy2} that it is $\Omega_{-}(X^{1/4} (\log X)^{1/4})$. These results have been improved several times, see \cite{berndt} for a detailed discussion.

In Theorem \ref{philrudn2}, in case $a)$ Phillips and Rudnick prove a $\Omega_{-}$-result, in case $b)$ the sign of the $\Omega$ depends on the group, whereas in case $c)$, the sign cannot be determined by their method. Although they do not mention anything for the sign of their $\Omega$-results, in case that the group $\G$ is as in Corollary \ref{lastcorollary}, their method implies a $\Omega_{\pm}$-result for $e(X;z,z)$. Our analysis allows us to prove a $\Omega_{\pm}$-result for $w \in B(z, \delta)$, which is one of the main reasons we choose to emphasize the sign of our $\Omega$-results. 
\end{remark}

\begin{remark} 
There are specific arithmetic groups for which we know they satisfy the conditions of Corollary \ref{lastcorollary}. 
If $\Gamma$ is a subgroup of $\slz$ of finite index, then it follows from \cite{huang},\cite{young} that in many
cases $\Gamma$ has sufficiently many cusp forms at $z$ for $z$ remaining in a compact subset of the surface.
Further, every $\G (N)$ with $N=5$ or $\geq 7$ has null vectors. For $\G_{0}(N)$, there are also groups having null vectors, for instance $\G_{0}(25)$. For further discussion on the null vectors of these arithmetic groups see \cite[p.~80-81]{petridis}, \cite[p.~151-153]{huxley}.

For $\Gamma = \slz$ and certain other groups it is conjectured that the real Satake parameters $t_j$ are linearly independent over $\mathbb{Q}$. Such a conjecture would allow to apply Kronecker's theorem \cite[p.~510, theorem 444]{hardy} and find a sequence of $R_m \to \infty$ such that the exponentials $\{ e^{it_j R_m} \}_{j=1}^{n}$  approach the point $-1$ simultaneously (see Lemma \ref{dirichletsboxprinciple} in section \ref{result1compactcase}). Using the fixed sign properties of the $\Gamma$-function from section \ref{lemmasproofs}, this would allow to substitute $\Omega_{-}$ and $\Omega_{+}$ in Theorems \ref{result1} and \ref{result2} with $\Omega_{\pm}$.
\end{remark}

\begin{remark}
In \cite{cramer}, Cram\'er studied the normalized error term of the Chebyshev's prime counting function $\psi(x)$. He proved that
 \begin{displaymath}
\frac{\psi(e^x)-e^x}{e^{x/2}}
 \end{displaymath}
has mean square average \cite[p.~148, eq.~(1)]{cramer}, whereas 
 \begin{displaymath}
\frac{\psi(u)-u}{u^{a}}
 \end{displaymath}
does not have mean square average for $a<1$ \cite[p.~148, eq.~(2)]{cramer}. For the hyperbolic lattice point problem Theorems \ref{philrudn1} and \ref{result1} show that a similar phenomenon appears for the error term $e(X;z,w)$. 
\end{remark}


\subsection{Acknowledgments} I would deeply like to thank my supervisor Y. Petridis for his guidance and encouragement. I would also like to thank G. Sakellaris for his valuable help on some technical details of this paper and the anonymous referee for his helpful comments.

\section{Lemmas}\label{lemmasproofs}

One of the key ingredients in the proofs of our results is the following lemma.
\begin{lemma}\label{gammalemma} For every $t \in \R$, we have:
\\
a) \begin{equation} \label{lemmagammaequation1}
\Re\left( \frac{\G(it)}{\G(3/2+it)} \right) < 0,
\end{equation}
b) \begin{equation} \label{lemmagammaequation2}
\Re\left( \frac{\G(it)}{\G(3/2+it) (1 +it)} \right) < 0.
\end{equation}
\end{lemma}
The proof of Lemma \ref{gammalemma} uses an elementary result about real functions.
\begin{lemma}\label{sublemma} Let $f : (-\infty , 0) \to \mathbb{R}$ be a continuous and strictly increasing real valued function such that $f(x) \sin(x)$ is integrable in $(-\infty, 0)$. Then 
 \begin{displaymath}
\int_{-\infty}^{0} f(x)  \sin(x) dx < 0.
\end{displaymath}
\end{lemma}
\begin{proof} (of Lemma \ref{sublemma})
 Since $f(x) \sin(x)$ is integrable in $(-\infty,0)$, we split the integral as 
\begin{eqnarray*} 
\int_{-\infty}^{0} f(x) \sin(x) dx &=& \sum_{n=0}^{\infty} \int_{-2 (n+1) \pi}^{-2n \pi} f(x) \sin(x) dx \\
&=&  \sum_{n=0}^{\infty}  \int_{-2n \pi - \pi}^{-2n \pi } \left( f(x) - f(x-\pi) \right) \sin(x) dx .
\end{eqnarray*} 
Since $f$ is strictly increasing and $\sin(x)$ is negative in the interval $(-2n\pi - \pi, -2n \pi)$\textcolor{red}{,} the statement follows. 
\end{proof}

\begin{proof} (of Lemma \ref{gammalemma})
$a)$ Since $\G(\overline{z}) = \overline{\G(z)}$, it suffices to prove the Lemma for $t >0$. Using \cite[p.~909, eq.~(8.384.1)]{gradry} we get
\begin{displaymath}
 \frac{\G(it)}{\G(3/2+it)} = \frac{2}{\sqrt{\pi}} B(it, 3/2),
\end{displaymath}
where $B(x,y)$ is the Beta function. By the definition of the Beta function \cite[p.~908, eq.~(8.380.1)]{gradry} and the formula
 \begin{displaymath}
B(x+1,y) = B(x,y) \frac{x}{x+y}
\end{displaymath}
we see that inequality (\ref{lemmagammaequation1}) is equivalent with
\begin{eqnarray*}
 \frac{2t}{3} \int_{0}^{1}\cos (t \log s ) (1-s)^{1/2} ds + \int_{0}^{1} \sin (t \log s) (1-s)^{1/2} ds < 0.
\end{eqnarray*}
Using integration by parts, setting $s=e^{x/t}$ and 
applying Lemma \ref{sublemma} for $f(x)= \frac{1}{t} (1-e^{x/t})^{1/2} e^{x/t}  - \frac{2}{3} ((1-e^{x/t})^{1/2} e^{x/t})' $, part $a)$ follows.
\\
$b)$ It follows in the same way. Using the formula

\begin{displaymath}
 \frac{\G(it)}{\G(3/2+it) (1 +it)} = -\frac{ i+t}{\sqrt{\pi} t (1+t^2)} B(1+it, 1/2),
\end{displaymath}
it follows that (\ref{lemmagammaequation2}) is equivalent with
\begin{eqnarray*}
-t \int_{0}^{1}\cos (t \log s ) (1-s)^{-1/2} ds + \int_{0}^{1} \sin (t \log s) (1-s)^{-1/2} ds < 0.
\end{eqnarray*}
\end{proof}
\begin{remark} \label{remarkforc}
 There exists a positive constant $c>0$ such that for $|t|>c$ we have
\begin{equation} \label{lemmagammaequation3}
\Re\left( \frac{\G(it)}{\G(3/2+it)}\frac{it}{ 1 +it} \right) < 0.
\end{equation}
This can be deduced easily from Stirling's formula. Using Lemma \ref{gammalemma} and working as in $a)$ we can get the desired result for $c \approx 2.30277...$, whereas in fact $c$ can be numerically found to be around $\approx 1.59135...$ using Mathematica. For this $c$, we can choose $a$ of Theorem \ref{result1} by $a =1/4+c^2 \approx 2.7823... $ (see section  \ref{result1compactcase}). In particular, $\slz$ satisfies the condition $\lambda_1>a$.
\end{remark} 

We will also use the following local Weyl's law for $\L(\GmodH)$ (see \cite[p.~86, Lemma~2.3]{phirud}).

\begin{theorem}[Local Weyl's law] \label{localweylslaw} For every $z$, as $T \to \infty$,
\begin{eqnarray*}
\sum_{|t_j| <T} |u_j(z)|^2 + \sum_{\frak{a}} \frac{1}{4\pi} \int_{-T}^{T} | E_{\frak{a}} \left( z, 1/2+it \right) |^2 dt \sim c T^2,
\end{eqnarray*}
where $c=c(z)$ depends only on the number of elements of $\Gamma$ fixing $z$.
\end{theorem}
When $z$ remains in a bounded region of $\mathbb{H}$ (more specifically in a compact set), the constant $c(z)$ is uniformly bounded, depending only on $\G$.

\begin{remark}
Phillips and Rudnick in \cite{phirud} generalize Theorem \ref{philrudn1} and case $a)$ of Theorem \ref{philrudn2} in the case of the $n$-dimensional hyperbolic space $\H^n$ \cite[p.~106]{phirud}. Considering the $n$-th dimensional analogues of Theorems \ref{result1}, \ref{result2} we notice that in order to make our method work for $\H^n$, we need stronger {\lq fixed sign\rq} properties for the $\G$-function. For instance, for Theorem \ref{result1} we would need the property that
\begin{eqnarray*}
\Re \left( \frac{\G(it)}{\G \left(\frac{n+1}{2} + it \right) (1+it)} \right)
\end{eqnarray*}
fixes sign for all $t \in \R$. However, this property fails for $n \geq 4$. 
\end{remark}


\section{$\Omega$-results for the averaged error term $M(X;z,w)$ } \label{result1compactcase}

\subsection{Proof of Theorem \ref{result1} for $\G$ cocompact} The quantity $N(X;z,w)$ can be interpreted as
\begin{equation} \label{interpretationofn}
N(X;z,w) =  K(z,w) = \sum_{\gamma \in \G} k(u (z, \gamma w) ),
\end{equation}
for $k(u) =  \chi_{[0, (X-2)/4]} (u )$. Because $k$ is not smooth, we cannot apply the pre-trace formula to the kernel $K(z,w)$. Instead of that, we work with $M(X;z,w)$. Using \cite[p.~321, eq.~(2.7)]{cham2}, the Selberg Harish-Chandra transform $h(t) = h_X(t)$ of $k(u)$ can be expressed as
\begin{equation}
h(t) = 2 \pi \sinh r P_{-1/2+it}^{-1} (\cosh r),
\end{equation}
where $r = \cosh^{-1}(X/2)$ and $ P_{\nu}^{\mu} (z)$ is the associated Legendre function. 
Using the formula \cite[p.~971, eq.~(8.776.1)]{gradry}, for $t \in \mathbb{R}$ we get
\begin{equation} \label{httransform}
h(t) = 2 \sqrt{\pi} \Re \left(  \frac{\G(it)}{\G \left(\frac{3}{2} +it \right)} X^{i t} \right) \left( X^{1/2} + O (X^{-3/2}) \right).
\end{equation} 
We first deal with $M(X;z,w)$ for the cocompact case. 
\begin{proof}
For $z$ fixed, consider a sequence of points $\{w_n\}_{n=1}^{\infty}$ such that $w_n \to z$. Then, for every $j$ we get
\begin{eqnarray*}
\overline{u_j(w_{n})} \to \overline{u_j(z)},
\end{eqnarray*}
as $n \to \infty$ (where we do not know uniformity in the limit). For $X= e^R$ we define 
\begin{eqnarray*}
 \widehat{M}(R;z,w) : = M(X;z,w).
\end{eqnarray*}
Using the spectral theorem for kernels (see \cite[pg.~104, Theorem~7.4)]{iwaniec}), Theorem \ref{mainformulaclassical}, equations (\ref{seconderrorterm}), (\ref{mtzdefinition}), (\ref{interpretationofn}), estimates about $h(t)$ (\cite[pg.~320, Lemma~2.4.b)]{cham2}) and the fact that $h(t)$ is an even function we get
\begin{equation} \label{almostspectralexpansionerror2}
M (X;z,w_n) =  \sum_{t_j>0 }  u_j(z) \overline{u_j(w_n)} \  \left(\frac{1}{X} \int_{2}^{X} \frac{h_x (t_j)}{x^{1/2}} d x \right)  + O \left(\sum_{1/2<s_j \leq 1}  \frac{1}{X} \int_{2}^{X} \frac{x^{1/2-s_j} }{2s_j-1} d x \right).
\end{equation}
Since the $s_j$'s are discrete, there exists a constant $\sigma = \sigma_{\G} \in (0, 1/2]$, depending only on $\G$, such that $s_j -1/2 \geq \sigma$ for all small eigenvalues. We conclude 
 \begin{equation} \label{almostspectralexpansionerror1}
\sum_{1/2<s_j \leq 1}  \frac{1}{X} \int_{2}^{X} \frac{x^{1/2-s_j} }{2s_j-1} d x  = O( X^{-\sigma} ).
\end{equation}
We use equation (\ref{httransform}) to obtain
\begin{equation}
\widehat{M}(R;z,w_n) = 2  \sqrt{\pi} \sum_{t_j  >0}  u_j(z)\overline{u_j(w_n)}  \Re \left(   \frac{\G(it_j)}{\G \left(\frac{3}{2} +it_j \right)} F(R, t_j)\right) + O(e^{-\sigma R}),
\end{equation}
where
\begin{equation} \label{usefulestimate}
 F(R, t_j) = e^{-R} \int_{2}^{e^R} x^{it_j} \left(1+ O \left(x^{-2} \right) \right) d x = \frac{e^{it_jR} }{1+it_j} +O \left(\frac{e^{-R}}{1+|t_j|} \right).
\end{equation}
Using Stirling's formula \cite[p.~895, eq.~(8.328.1)]{gradry} and Theorem \ref{localweylslaw} we see that
\begin{eqnarray*}
\widehat{M}(R;z,w_n) = 2  \sqrt{\pi} \sum_{t_j >0}  u_j(z)\overline{u_j(w_n)} \Re \left( \frac{\G(it_j)}{\G \left(\frac{3}{2} +it_j \right) (1+ it_j)}  e^{ i t_j R} \right)
+ O(e^{- \sigma R}).
\end{eqnarray*}
 For $A>1$, we split the sum in the intervals $[0, A)$ and $[A, +\infty)$.  Stirling's formula, Theorem \ref{localweylslaw} and Cauchy-Schwarz inequality imply the bound 
\begin{eqnarray*}
\sum_{t_j \geq A}   u_j(z)\overline{u_j(w_n)}  \Re \left( \frac{\G(it_j)}{\G \left(\frac{3}{2} +it_j \right) (1+ it_j)} e^{i t_j R} \right) =O(A^{-1/2}).
\end{eqnarray*}
Let $\epsilon_1 >0$. Since for every $j$ we have $\overline{u_j(w_{n})} \to \overline{u_j(z)}$,
we can find an integer $n_0 = n_0 (\epsilon_1, A)$ such that 
\begin{eqnarray*}
\overline{u_j(w_{n})} = \overline{u_j(z)} +O( \epsilon_1)
\end{eqnarray*}
for every $n \geq n_0$ and for every $j$ such that $0<t_j<A$. Thus, using Theorem \ref{localweylslaw} and Cauchy-Schwarz inequality, for $n \geq  n_0 (\epsilon_1, A)$ we get
\begin{eqnarray*}
 \widehat{M}(R;z,w_n)  &=& 2   \sqrt{\pi} \sum_{0< t_j < A }    |u_j(z)|^2 \Re \left( \frac{\G(it_j)}{\G \left(\frac{3}{2} +it_j \right) (1+ it_j)}  e^{it_j R}  \right) \\
&&+ O \left(  A^{-1/2} + \epsilon_1 +e^{- \sigma R} \right).
\end{eqnarray*}
The sum for $t_j<A$ can be handled by applying Dirichlet's principle (see \cite[p.~96, Lemma 3.3]{phirud}).
\begin{lemma} [Dirichlet's box principle] \label{dirichletsboxprinciple} Let $r_1, r_2, ... , r_n$ be $n$ distinct real numbers and $M>0$, $T>1$. Then, there is an $R$, $M \leq R \leq M T^n$, such that
\begin{eqnarray*}
|e^{ir_jR} -1| < \frac{1}{T}
\end{eqnarray*}
for all $j=1,...,n$.
\end{lemma}
We apply Dirichlet's principle to the sequence $e^{it_jR}$. For any $M>0$ and any $T>1$ sufficiently large we find an $R$ such that
\begin{eqnarray*}
 \widehat{M}(R;z,w_n)  &=&  2  \sqrt{\pi}   \sum_{0< t_j < A }   |u_j(z)|^2 \Re \left( \frac{\G(it_j)}{\G \left(\frac{3}{2} +it_j \right) (1+ it_j)}\right) \\
&&+ O \left( T^{-1} + A^{-1/2} + \epsilon_1 +e^{- \sigma R} \right).
\end{eqnarray*}
We apply local Weyl's law and Lemma \ref{gammalemma} to the last sum. Local Weyl's law implies that as $A \to \infty$ the sum remains bounded and, for $\G$ cocompact, there exist infinitely many $j$'s such that $u_j(z) \neq 0$. Lemma \ref{gammalemma} implies that all the nonzero terms are negative. Hence, there exists an $A_0$ such that for every $A \geq A_0$:
\begin{equation} \label{finalsumlowerbound}
 \left| \sum_{0< t_j < A}  |u_j(z)|^2 \Re \left( \frac{\G(it_j)}{\G \left(\frac{3}{2} +it_j \right) (1+ it_j)}\right) \right|  \gg 1.
\end{equation}
 Choosing $T$ sufficiently large, $A$ fixed and sufficiently large and $\epsilon_1$ fixed and sufficiently small, we deduce that we can choose $n_0$ fixed such that $M(X;z,w_n) = \Omega_{-}(1)$ for every $n \geq n_0$. Hence, $M(X;z,w) = \Omega_{-}(1)$ for $w$ in a fixed $\delta$-neighbourhood of $z$.
This proves case $a)$ of Theorem \ref{result1}.

 To prove that if $\lambda_1>a$ the limit does not exist, we consider the finite sum
\begin{eqnarray*}
 S_{z,A}(R) = 2  \sqrt{\pi} \sum_{0< t_j < A }    |u_j(z)|^2 \Re \left( \frac{\G(it_j)}{\G \left(\frac{3}{2} +it_j \right) (1+ it_j)}  e^{it_j R}  \right)
\end{eqnarray*}
for $A$ chosen finite and sufficiently large.
We first prove that it attains at least two different values. We differentiate $S_{z,A}(R)$. Since $A$ is finite, we compute
\begin{eqnarray*}
 \frac{\partial S_{z, A}}{\partial R} (R) =  2 \sqrt{\pi} \sum_{0< t_j < A }    |u_j(z)|^2 \Re \left( \frac{\G(it_j)}{\G \left(\frac{3}{2} +it_j \right)} \frac{it_j}{(1+ it_j)}  e^{it_j R}  \right).
\end{eqnarray*}
Applying again Dirichlet's principle we find a sufficiently large $T_0$ and a $R_0$, depending on $T_0$, such that 
\begin{eqnarray*}
 \frac{\partial S_{z, A}}{\partial R} (R_0) = 2  \sqrt{\pi} \sum_{0< t_j < A }    |u_j(z)|^2 \Re \left( \frac{\G(it_j)}{\G \left(\frac{3}{2} +it_j \right)} \frac{it_j}{(1+ it_j)} \right)  + O(A^{1/2}  T_0^{-1}).
\end{eqnarray*}
Assume $t_1 >c$, with $c$ as in Remark \ref{remarkforc} (hence $\lambda_1 > 1/4 +c^2 = a$). We conclude that $\frac{\partial S_{z, A}}{\partial R} (R_0) \neq 0$, hence $S_{z,A}(R)$ is not constant. In particular, it admits at least two different values $B_1, B_2$. Assume we express $B_{\nu}$ as 
\begin{eqnarray*}
B_{\nu} =  2  \sqrt{\pi} \sum_{0< t_j < A }    |u_j(z)|^2 \Re \left( \frac{\G(it_j)}{\G \left(\frac{3}{2} +it_j \right) (1+ it_j)} e^{it_j R_{\nu}}  \right) 
\end{eqnarray*}
for $\nu=1,2$.  We estimate
\begin{eqnarray*}
\left|  \widehat{M}(R;z,w_n)  - B_{\nu} \right| &\ll&   \sum_{0< t_j < A }    |u_j(z)|^2  \left| \Re \left( \frac{\G(it_j)}{\G \left(\frac{3}{2} +it_j \right) (1+ it_j)} \left(e^{it_j (R-R_{\nu})} - 1 \right) \right) \right| \\
&& + O \left( A^{-1/2} + \epsilon_1 +e^{- \sigma R} \right).
\end{eqnarray*}
Applying Dirichlet's principle, for $\nu=1,2$ we find sequences $T_{\mu,\nu} \to \infty$ and sequences $R_{\mu,\nu} \to \infty$ as $\mu \to \infty$ such that for all $t_j \in (0,A)$:
\begin{eqnarray*}
e^{it_j (R_{\mu,\nu} -R_{\nu})}  = 1 + O(T_{\mu,\nu}^{-1}).
\end{eqnarray*}
Hence, we conclude that 
\begin{eqnarray*}
\left|  \widehat{M}(R_{\mu, \nu};z,w_n)  - B_{\nu} \right| = O \left(T_{\mu,\nu}^{-1} + A^{-1/2} + \epsilon_1 +e^{- \sigma R_{\mu,\nu}} \right).
\end{eqnarray*}
Since $R_{\mu, \nu} \to \infty$, we conclude that $M(X,z,w)$ approaches both values $B_1, B_2$ infinitely many times as close as we want as $X \to \infty$\textcolor{red}{.} Since $B_1 \neq B_2$, we conclude that the $M(X,z,w)$ does not have a limit as $X \to \infty$.
\end{proof}

\begin{remark}
In order to prove the lower bound (\ref{finalsumlowerbound}), it is enough to assume that there exists at least one $j$ such that $u_j(z) \neq 0$. Thus, in any such case, the contribution of the discrete spectrum in $M(X;z,w_n)$ is $\Omega_{-}(1)$ for $n \geq n_0$. The same argument holds for the last statement of Theorem \ref{result1}.
\end{remark}

\subsection{Proof of Theorem \ref{result1} for $\G$ cofinite} \label{result1cofinitecase}

For $\G$ cofinite but not cocompact, the hyperbolic Laplacian $-\Delta$ has also continuous spectrum which corresponds to the Eisenstein series $E_{\frak{a}}(z, 1/2+it)$ (see \cite[chapters 3, 6 and 7]{iwaniec}). To deal with the cofinite case (part $b)$ of Theorem \ref{result1}, we have to study the contribution of the continuous spectrum in $M(X;z,w_n)$.  

\begin{proof} Let $\G$ be cofinite but not cocompact. Working as in the proof of the cocompact case we get
\begin{eqnarray*}
M (X;z,w_n) &=&  \sum_{t_j>0 } u_j(z)  \overline{u_j(w_{n})}  \left(\frac{1}{X} \int_{2}^{X} \frac{h_x(t_j)}{x^{1/2}} d x \right) + O(X^{-\sigma})   \\
&&+ \sum_{\mathfrak{a}} \frac{1}{4 \pi} \int_{-\infty}^{\infty}  E_{\mathfrak{a}}(z,1/2+it)   \overline{E_{\mathfrak{a}}(w_n,1/2+it)}  \left( \frac{1}{X} \int_{2}^{X} \frac{h_x(t)}{x^{1/2}} d x \right) dt, 
\end{eqnarray*}
where the second sum is over the cusps $\mathfrak{a}$ of $\G$. Hence, the contribution of the continuous spectrum for the cusp $\frak{a}$ in $M(X;z,w_n)$ is equal to
\begin{equation} \label{continuousaverage}
\frac{1}{4 \pi} \int_{-\infty}^{\infty}  E_{\mathfrak{a}}(z,1/2+it)   \overline{E_{\mathfrak{a}}(w_n,1/2+it)}    \left( \frac{1}{X} \int_{2}^{X} \frac{h_x(t)}{x^{1/2}} d x \right) dt.
\end{equation}
 Set 
\begin{eqnarray*}
A(X) = \int_{-\infty}^{\infty} \frac{1}{X} \int_{2}^{X} \frac{h_x(t)}{x^{1/2}} d x dt.
\end{eqnarray*}
We have the following lemma, which is analogous to Lemma 2.4 in \cite{phirud}. 
\begin{lemma}\label{selbergtransformconvergence}
As $X \to \infty$ we have
\begin{eqnarray*}
\lim_{X \to \infty} A(X) = 4 \pi.
\end{eqnarray*}
\end{lemma}
\begin{proof}
The Selberg/Harish-Chandra transform $h(t)$ of $\chi_{[0, (\cosh r -1)/2]}$ can be written in the form 
\begin{eqnarray*}
h(t) = 2 \sqrt{2} \int_{-\infty}^{\infty} e^{itu} (\cosh r - \cosh u)^{1/2} \chi_{[-r, r]}(u) du,
\end{eqnarray*}
see \cite[p.~84, 85, eq.~(2.9), (2.10)]{phirud}. Hence
\begin{eqnarray*}
A(X) =  \int_{-\infty}^{\infty} \int_{-\infty}^{\infty} e^{itu} \Phi_X(u) du dt,
\end{eqnarray*}
where $\Phi_X(u)$ is given by
\begin{eqnarray*}
\Phi_X(u) = \frac{4}{X} \int_{|u|}^{ \cosh^{-1}(X/2)} \sinh r \sqrt{1 - \frac{\cosh u}{\cosh r}} d r.
\end{eqnarray*}
Using the Fourier inversion formula and easy estimates we get
\begin{eqnarray*}
A(X) =  2 \pi \Phi_X(0) =  \frac{8 \pi}{X} \left(\frac{X}{2}-1\right) + O \left(\frac{\log X}{X} \right) \to  4 \pi.
\end{eqnarray*}
\end{proof}
Let $\phi_{\frak{a},n}(t), \phi_{\frak{a}}(t)$ be defined as:
\begin{eqnarray} \label{phifunctions}
\phi_{\frak{a},n}(t) &=&  E_{\mathfrak{a}}(z,1/2+it)   \overline{E_{\mathfrak{a}}(w_n,1/2+it)}  -  |E_{\mathfrak{a}}(z,1/2)|^2, \\
\nonumber \phi_{\frak{a}}(t) &=&  |E_{\mathfrak{a}}(z,1/2+it)|^2 -  |E_{\mathfrak{a}}(z,1/2)|^2.
\end{eqnarray}
 Thus, the contribution of cusp $\frak{a}$ to (\ref{continuousaverage}) can be written in the form
\begin{equation}\label{cuspacontributionexplicitly}
 \frac{1}{4 \pi}  |E_{\mathfrak{a}}(z,1/2)|^2 A(X) +  \frac{1}{4 \pi} \int_{-\infty}^{\infty} \phi_{\frak{a},n}(t)  \left( \frac{1}{X} \int_{2}^{X} \frac{h(t)}{x^{1/2}} dx \right) dt.
\end{equation}
Using equation (\ref{httransform}), the second summand of (\ref{cuspacontributionexplicitly}) takes the form
\begin{eqnarray}
\frac{1}{2 \sqrt{\pi}}\int_{-\infty}^{\infty} \phi_{\frak{a},n}(t) \Re \left( \frac{\G(it)}{\Gamma(3/2 +it)} \frac{1}{1+it} X^{it} \right) dt &+& O \left( \frac{1}{X} \int_{-\infty}^{\infty} \phi_{\frak{a},n}(t)  \Re \left( \frac{\G(it)}{\Gamma(3/2 +it)} \frac{2^{it}}{1+it} \right)  dt \right) \nonumber \\ \label{secondsummand}
&+& O \left( \frac{1}{X^2} \int_{-\infty}^{\infty} \phi_{\frak{a},n}(t) \Re \left( \frac{\G(it)}{\Gamma(3/2 +it)} \frac{1}{1-it} X^{it} \right)  dt \right)
\end{eqnarray}
For any $A>1$  we split the first integral:
\begin{equation}\label{firsttermsplit}
\int_{-A}^{A} \phi_{\frak{a},n}(t) \Re \left( \frac{\G(it)}{\Gamma(3/2 +it)} \frac{1}{1+it} X^{it} \right) dt 
+\int_{|t| > A} \phi_{\frak{a},n}(t) \Re \left( \frac{\G(it)}{\Gamma(3/2 +it)} \frac{1}{1+it} X^{it} \right) dt.
\end{equation}
Since $w_n \to z$, using Theorem \ref{localweylslaw} and Cauchy-Schwarz inequality, the integral for $|t|>A$ is bounded independently of $n$:
\begin{equation}\label{firsttermsplitbound}
\int_{|t| > A} \phi_{\frak{a},n}(t) \Re \left( \frac{\G(it)}{\Gamma(3/2 +it)} \frac{1}{1+it} X^{it} \right) dt = O (A^{-1/2}).
\end{equation}
In $[-A,A]$, we approximate $\phi_{\frak{a},n}(t)$: for every $\epsilon_1 >0$ there exists a $n_0 = n_0(\epsilon_1, A)$ such that for every $n \geq n_0$: 
\begin{eqnarray*}
\phi_{\frak{a},n}(t) = \phi_{\frak{a}}(t) +O (\epsilon_1)
\end{eqnarray*}
for every $t \in [-A,A]$. Thus, we get
\begin{eqnarray*}
\int_{-A}^{A} \phi_{\frak{a},n}(t) \Re \left( \frac{\G(it)}{\Gamma(3/2 +it)} \frac{1}{1+it} X^{it} \right) dt &=&
\int_{-\infty}^{\infty} \phi_{\frak{a}}(t) \Re \left( \frac{\G(it)}{\Gamma(3/2 +it)} \frac{1}{1+it} X^{it} \right) dt \\
&-&\int_{|t|>A} \phi_{\frak{a}}(t) \Re \left( \frac{\G(it)}{\Gamma(3/2 +it)} \frac{1}{1+it} X^{it} \right) dt \\
&+& O \left( \epsilon_1 \int_{-A}^{A} \Re \left( \frac{\G(it)}{\Gamma(3/2 +it)} \frac{1}{1+it} X^{it} \right) dt \right) . 
\end{eqnarray*}
 Using the bound  $\phi_{\frak{a}}(t) = O(t)$ for small $t$  and Theorem \ref{localweylslaw} for $t \to \infty$ we get that the function 
$$\theta_{\frak{a}}(t) = \phi_{\frak{a}}(t) \frac{\G(it)}{\Gamma(3/2 +it)} \frac{1}{1+it}$$
is in $L^1(\mathbb{R})$. Applying the Riemann-Lebesgue Lemma to the first term we conclude that it converges to $0$ as $X \to \infty$.
We work as for (\ref{firsttermsplitbound}) to see that the second term is bounded by $O (A^{-1/2})$. Using that the function
$$g_{\frak{a}}(t) = \Re \left( \frac{\G(it)}{\Gamma(3/2 +it)} \frac{1}{1+it} X^{it} \right)$$
is in $L^1(\mathbb{R})$ uniformly in $X$, we see that the third term is $O(\epsilon_1)$. Using trivial estimates instead of the Riemann-Lebesgue Lemma in (\ref{secondsummand}) for the $O$-terms, we conclude that for every $\epsilon_1>0, A>1$ there exists a $n_0 = n_0(\epsilon_1, A)$ such that for every $n \geq n_0$:
\begin{eqnarray*} 
\frac{1}{4 \pi} \int_{-\infty}^{\infty} \phi_{\frak{a},n}(t)  \left( \frac{1}{X} \int_{2}^{X} \frac{h(t)}{x^{1/2}} dx \right) dt = O (\epsilon_1 +A^{-1/2}) + o(1).
\end{eqnarray*}
Thus, choosing $\epsilon_1 = A^{-1/2}$ we conclude that for every $\epsilon_1>0$ there exists a $n_0 = n_0(\epsilon_1)$ such that for every $n \geq n_0$ the contribution of the continuous spectrum to $M(X;z,w_n)$ is equal to 
\begin{equation} \label{continuouscontributiozequalw}
 \sum_{\mathfrak{a}} \left|E_{\mathfrak{a}}(z,1/2) \right|^2 + O(\epsilon_1) + o(1).
\end{equation}
Case $b)$ of Theorem \ref{result1} follows for $\epsilon_1$ sufficiently small and fixed. 
\end{proof}
When there is no contribution from the eigenvalues $\lambda_j>1/4$, the following proposition follows immediately from  (\ref{continuouscontributiozequalw}).
\begin{proposition} If either $\G$ does not have eigenvalues $\lambda_j>1/4$ or $u_j(z) = 0$ for every $\lambda_j> 1/4$, then we have
\begin{displaymath}
\lim_{X \to \infty} M(X;z,z) = \sum_{\mathfrak{a}} \left|E_{\mathfrak{a}}(z,1/2) \right|^2.
\end{displaymath}
\end{proposition}

\section{$\Omega$-results for the error term $e(X;z,w)$} \label{theclassiclaproblemerrorterm}

We now prove Theorem \ref{result2}. Assume that $\G$ is cofinite but not cocompact. In order to deal with the error term $e(X;z,w)$ we mollify it as in Phillips and Rudnick \cite{phirud}. Let $\psi \geq 0$ be a smooth even function compactly supported in $[-1,1]$,  such that 
\begin{eqnarray}
\int_{-\infty}^{+\infty} \psi(x) e^{-itx} dx  := \hat{\psi} (t) \geq 0
\end{eqnarray}
and $\int_{-\infty}^{\infty} \psi(x) dx = 1$. For every $\epsilon >0$ we also define the family of functions $\psi_{\epsilon}(x) = \epsilon^{-1} \psi(x/\epsilon)$. We have $0 \leq \hat{\psi}_{\epsilon}(x) \leq 1$ and $\hat{\psi}_{\epsilon}(0) = 1$. We study the contribution of the discrete spectrum first. 

\subsection{The contribution of the discrete spectrum} For $z$ fixed, we pick again a sequence $\{w_n\}_{n=1}^{\infty}$ converging to $z$. For every $n \geq 1$ we define
\begin{equation} \label{quantitynxzk}
\widehat{e}_{n} (R,z)  =  \frac{e(e^R;z,w_{n})}{e^{R/2}},
\end{equation}
and we consider the convolution
\begin{eqnarray*}
\widehat{e}_{n, \epsilon} (R,z) = \left(\psi_{\epsilon} \ast \widehat{e}_{n} (\cdot,z) \right) (R) := \int_{-\infty}^{\infty} \psi_{\epsilon} (R-Y) \widehat{e}_{n} (Y,z) dY.
\end{eqnarray*}
In order to prove a lower bound for $\widehat{e}_{n} (R,z) $ it suffices to prove a lower bound for $\widehat{e}_{n, \epsilon} (R,z)$.

Using the pre-trace formula for $\L(\GmodH)$ (\cite[pg.~104, Theorem~7.4)]{iwaniec}), the expression (\ref{httransform}), the bound 
\begin{equation} \label{usefulboundpsiepsilon}
\hat{\psi}_{\epsilon}(t_j) = O_k \left( (\epsilon|t_j|)^{-k} \right)
\end{equation} 
for every $k \in \mathbb{N}$ and working as in \cite{phirud} we conclude that the contribution of the discrete spectrum in $\widehat{e}_{n, \epsilon} (R,z)$ is equal to: 
\begin{equation} \label{disccontr1}
2 \sqrt{\pi}  \sum_{t_j>0 } u_j(z) \overline{u_j(w_{n})} \Re \left( \frac{\G(it_j)}{\G \left(\frac{3}{2} +it_j \right)}   e^{it_j R} \right) \hat{\psi}_{\epsilon} (t_j)  + O \left( e^{- \sigma R} \right).
\end{equation}
Let $A>1$. Clearly for $k \in \mathbb{N}$ we have $k > 1/2$ and thus, using the estimate (\ref{usefulboundpsiepsilon}), we can bound the tail of the series for $t_j > A$. Applying (\ref{usefulboundpsiepsilon}), Theorem \ref{localweylslaw} and Stirling's formula, we conclude that (\ref{disccontr1}) takes the form
\begin{equation} \label{disccontr2}
 2 \sqrt{\pi}  \sum_{0<t_j <A }  u_j(z) \overline{u_j(w_{n})} \Re \left( \frac{\G(it_j)}{\G \left(\frac{3}{2} +it_j \right)}   e^{it_j R} \right)  \hat{\psi}_{\epsilon} (t_j) + O \left( A^{1/2-k} \epsilon^{-k} +  e^{- \sigma R} \right).
\end{equation}
Let $\epsilon_1 >0$. We find again an integer $n_0 = n_0 (\epsilon_1, A)$ such that 
\begin{eqnarray*}
\overline{u_j(w_{n})} = \overline{u_j(z)} +O( \epsilon_1)
\end{eqnarray*}
for every $n \geq n_0$ and for every $j$ such that $0<t_j<A$.
Since $\hat{\psi}_{\epsilon}(x) $ is bounded, Cauchy-Schwarz inequality, weak Weyl's law (i.e. $\{ j : |t_j| \leq T \} \ll T^2$) and Theorem \ref{localweylslaw} yield that the quantity in (\ref{disccontr2}) is
\begin{eqnarray*}
2  \sqrt{\pi} \sum_{0<t_j <A }    |u_j(z)|^2   \Re \left( \frac{\G(it_j)}{\G \left(\frac{3}{2} +it_j \right)}   e^{it_j R} \right)   \hat{\psi}_{\epsilon} (t_j) +  O \left( A^{1/2-k} \epsilon^{-k} +  A^{1/2} \epsilon_1 +  e^{- \sigma R} \right). 
\end{eqnarray*}
Applying Dirichlet's principle \label{dirichletslemma} for the exponentials $e^{it_j R}$, for any $T>1$ sufficiently large we find an $R$ such that $e^{it_jR} = 1 +O(T^{-1})$,
thus concluding that the contribution of the discrete spectrum to $\widehat{e}_{n, \epsilon} (R,z)$ takes the form 
\begin{eqnarray*}
2 \sqrt{\pi}  \sum_{0<t_j <A }  |u_j(z)|^2    \Re \left( \frac{\G(it_j)}{\G \left(\frac{3}{2} +it_j \right)}   \right) \hat{\psi}_{\epsilon} (t_j) + O \left( T^{-1} A^{1/2} +  A^{1/2-k} \epsilon^{-k} + A^{1/2} \epsilon_1 +  e^{- \sigma R} \right).
\end{eqnarray*}
By part a) of Lemma \ref{gammalemma}, the coefficients in the sum are all negative, whereas the balance $\epsilon^{-1} = A^{1-3/(2k+2)}$, $\epsilon_1 = A^{-1/2}\epsilon $ implies that the error term is $O((T^{-1} A^{1/2} + \epsilon + e^{- \sigma R})$. For the function $\psi$, there exists one $\tau \in (0,1)$ such that $\hat{\psi}(x) \geq 1/2$ whenever $|x| \leq \tau$. Using this, local Weyl's law and the fact that $\hat{\psi}_{\epsilon}(t_j) = \hat{\psi} (\epsilon t_j)$ we bound the modulus of the above main term from below by 
\begin{eqnarray*} 
 \sum_{0<t_j <A }   |u_j(z)|^2 \hat{\psi}_{\epsilon} (t_j)  \left| \Re \left( \frac{\G(it_j)}{\G \left(\frac{3}{2} +it_j \right)}   \right)     \right| &\gg& \sum_{0<t_j < \tau/\epsilon}  |u_j(z)|^2  \left| \Re \left( \frac{\G(it_j)}{\G \left(\frac{3}{2} +it_j \right)}   \right) \right|.
\end{eqnarray*}
If $\G$ has sufficiently many cusp forms, we obtain the bound
\begin{eqnarray*}
 \sum_{0<t_j < \tau/\epsilon}  |u_j(z)|^2  \left| \Re \left( \frac{\G(it_j)}{\G \left(\frac{3}{2} +it_j \right)}   \right) \right|    \gg \epsilon^{-1/2}.
\end{eqnarray*}
Since the error is $O(T^{-1} A^{1/2} + \epsilon + e^{- \sigma R})$, there exists a fixed, sufficiently small $\epsilon_0>0$ such that, for $T$ and $R$ sufficiently large, $\epsilon_0^{-1/2}$ dominates the error. Therefore there exists a fixed integer $n_0 = n_0(\epsilon_0)$ such that for every $n \geq n_0$ the contribution of the discrete spectrum in $\widehat{e}_{n, \epsilon_0} (R,z)$ is $\Omega_{-}(1)$, i.e. for every $n \geq n_0$ the contribution in $e (X; z, w_{n})$ is $\Omega_{-}(X^{1/2})$.


\subsection{The contribution of the continuous spectrum} We have to consider the contribution of the continuous spectrum in $\widehat{e}_{n, \epsilon} (R,z) $, which is given by 
\begin{equation} \label{cofinitecasecontinuouscontribution}
 \sum_{\frak{a}} \frac{1}{4 \pi} \int_{-\infty}^{\infty}  E_{\mathfrak{a}}(z,1/2+it) \overline{E_{\mathfrak{a}}(w_{n},1/2+it)} \left(\int_{-\infty}^{\infty} \psi_{\epsilon} (R-Y)  \frac{h_{e^Y} (t)}{e^{Y/2}} dY \right) dt.
\end{equation}
Let $\phi_{\frak{a},n}(t)$ be as in equation (\ref{phifunctions}). Hence, for $h(t) = h_{e^Y}(t)$ the contribution of cusp $\frak{a}$ in (\ref{cofinitecasecontinuouscontribution}) takes the form
\begin{equation} \label{secondsummandofcuspa}
 \frac{1}{4 \pi}   |E_{\mathfrak{a}}(z,1/2)|^2 \int_{-\infty}^{\infty}  \left(\int_{-\infty}^{\infty} \psi_{\epsilon} (R-Y)  \frac{h(t)}{e^{Y/2}} dY \right) dt + \frac{1}{4 \pi} \int_{-\infty}^{\infty} \phi_{\frak{a},n}(t)  \left(\int_{-\infty}^{\infty} \psi_{\epsilon} (R-Y)  \frac{h(t)}{e^{Y/2}} dY \right) dt.
\end{equation}
Using \cite[p.~98,eq.~(3.30)]{phirud} we get
\begin{eqnarray*}
 \frac{1}{4 \pi} |E_{\mathfrak{a}}(z,1/2)|^2 \int_{-\infty}^{\infty}  \left(\int_{-\infty}^{\infty} \psi_{\epsilon} (R-Y)  \frac{h(t)}{e^{Y/2}} dY \right) dt =  |E_{\mathfrak{a}}(z,1/2)|^2   + O(e^{-R}).
\end{eqnarray*}
Using that $\psi(x)$ has support in $[-1,1]$ and quoting expansion (\ref{httransform}) 
we see that the second summand of (\ref{secondsummandofcuspa}) takes the form
\begin{equation} \label{phiacontribution}
\frac{1}{2 \sqrt{\pi}}  \Re \left( \int_{-\infty}^{\infty} \phi_{\frak{a},n}(t)  \frac{\G(it)}{\G(3/2+it)} e^{itR} \hat{\psi}_{\epsilon} (t) dt \right) 
+  O \left(   e^{-2R}  \int_{-\infty}^{\infty} \phi_{\frak{a},n}(t)   \Re \left( \frac{\G(it)}{\G(3/2+it) (-2+it)} \right)   dt   \right).
\end{equation}
 For the first term of (\ref{phiacontribution}), we work as in subsection \ref{result1cofinitecase}: we split the integral for $t \in [-A,A]$ and $|t|> A$. For $|t|>A$ we apply Theorem \ref{localweylslaw} and the bound (\ref{usefulboundpsiepsilon}) to get 
\begin{eqnarray*}
\Re \left( \int_{|t|>A} \phi_{\frak{a},n}(t)  \frac{\G(it)}{\G(3/2+it)} e^{itR} \hat{\psi}_{\epsilon} (t) dt \right) = O(\epsilon^{-1} A^{-1/2}),
\end{eqnarray*}
 independently of $n$. We approximate $\phi_{\frak{a},n}(t)$ uniformly $\epsilon_1$-close to $\phi_{\frak{a}}(t)$:    
$
\phi_{\frak{a},n}(t) = \phi_{\frak{a}}(t) + O(\epsilon_1)$ for every $n \geq n_0 = n_0(\epsilon_1)$ and for every $t \in [-A,A]$. The function $\phi_{\frak{a}}(t)$ satisfies the bounds $\phi_{\frak{a}}(t) = O(t)$ for small $t$ and $O(t^2)$ for large $t$. Using $\hat{\psi}_{\epsilon}(0)=1$, $\hat{\psi}_{\epsilon} (t) = O( (\epsilon |t|)^{-k})$) we deduce that, for any fixed $\epsilon>0$, the function 
\begin{eqnarray*}
 \phi_{\frak{a}}(t)  \frac{\G(it)}{\G(3/2+it)}\epsilon^2 \hat{\psi}_{\epsilon} (t) 
\end{eqnarray*}
is in $L^1(\mathbb{R})$ independently of $\epsilon$. Applying the Riemann-Lebesgue Lemma and local Weyl's law we deduce that 
\begin{eqnarray*}
\frac{1}{2 \sqrt{\pi}}  \Re \left( \int_{-A}^{A} \phi_{\frak{a}}(t)  \frac{\G(it)}{\G(3/2+it)} e^{itR} \hat{\psi}_{\epsilon} (t) dt \right)  
&=& \epsilon^{-2} G(R) + O(\epsilon^{-1} A^{-1/2}),
\end{eqnarray*}
with $G(R) = o(1)$, independently of $\epsilon$. Since the function
\begin{eqnarray*}
 \Re \left( \frac{\G(it)}{\G(3/2+it)} e^{itR} \right) =  \Re \left(\frac{\G(it)}{\G(3/2+it)}  \right)  \cos (tR)  - \Im \left( \frac{\G(it)}{\G(3/2+it)} \right)  \sin(tR)
\end{eqnarray*}
is bounded as $t \to 0$ we conclude it is in $L^1(\mathbb{R})$, hence 
\begin{eqnarray*}
\frac{1}{2 \sqrt{\pi}}  \Re \left( \epsilon_1 \int_{-A}^{A}  \frac{\G(it)}{\G(3/2+it)} e^{itR} \hat{\psi}_{\epsilon} (t) dt \right) = O(\epsilon_1).
\end{eqnarray*}
Working similarly we get
\begin{eqnarray*}
 e^{-2R}  \int_{-\infty}^{\infty} \phi_{\frak{a},n}(t)   \Re \left( \frac{\G(it)}{\G(3/2+it) (-2+it)} \right)   dt = O( e^{-2R})
\end{eqnarray*}
uniformly, i.e. independently of $n$.
Using the balance $\epsilon^{-2} = A^{1/2}$ and $\epsilon_1 = \epsilon$, we conclude that the contribution of the continuous spectrum in $\widehat{e}_{n, \epsilon} (R,z)$ can be finally written in the form
\begin{eqnarray*}
 \sum_{\frak{a}} |E_{\mathfrak{a}}(z,1/2)|^2  + \epsilon^{-2} G(R) + O(\epsilon + e^{-R}).
\end{eqnarray*}

\subsection{Proof of part a) of Theorem \ref{result2} } 

 Since $\G$ has suffieciently many cusp forms, the contribution of the discrete spectrum in $\widehat{e}_{n, \epsilon} (R,z)$ is of the form $k(\epsilon) + O(T^{-1} A^{1/2} + \epsilon + e^{-\sigma R})$ with $k(\epsilon) = \Omega_{-} ( \epsilon^{-1/2})$. Pick a fixed $\epsilon_0$ such that $\epsilon_0^{-1/2}$ dominates the $\sum_{\frak{a}} |E_{\mathfrak{a}}(z,1/2)|^2$ and the $O(\epsilon_0)$-terms. Let $T, R \to \infty$. Since $\epsilon_0$ is fixed, $\epsilon_0^{-2} G(R) \to 0$ as $R \to \infty$. Thus, for every $n \geq n_0 (\epsilon_0)$:
\begin{eqnarray*}
\widehat{e}_{n, \epsilon_0} (R,z) = \Omega_{-}(1),
\end{eqnarray*}
which implies $e(X;z, w_n) = \Omega_{-} (X^{1/2})$ for every $n \geq n_0$ .


\subsection{Proof of part b) of Theorem \ref{result2} } 

Assume that $\G$ has null-vectors. We have to prove that $e(X;z,w) = \Omega_{+}(X^{1/2})$ for $w$ in a small neighboorhood $B(z,\delta)$ of $z$. By Theorem \ref{philrudn1} of Phillips and Rudnick in \cite{phirud} we have $e(X;z,z) = \Omega_{+}(X^{1/2})$. Hence, in order to prove part $b)$, it suffices to prove the following Proposition.

\begin{proposition} If $\G$ has null-vectors, then there exists a $\delta = \delta_{\Gamma,z}>0$ such that for every $w \in B(z,\delta)$
 \begin{equation} \label{mnrt3}
\lim_{T \to \infty} \frac{1}{T} \int_{0}^{T}  \frac{e(e^r;z,w)}{e^{r/2}} dr > \frac{1}{2}  \sum_{\mathfrak{a}} \left|E_{\mathfrak{a}}(z,1/2) \right|^2.
\end{equation}
\end{proposition}
Since the proof is a routine using the ideas used in the proof of Theorem \ref{philrudn1} in \cite{phirud} and in section \ref{result1compactcase}, we only sketch the basic steps.
\begin{proof}(sketch) For any $w_n \to z$, the contribution of Maass forms in 
\begin{eqnarray*}
 \frac{1}{T} \int_{0}^{T}  \frac{e(e^r;z,w_n)}{e^{r/2}} dr
\end{eqnarray*} 
 is estimated using expression (\ref{httransform}), the local Weyl's law and the Cauchy-Schwarz inequality:
\begin{eqnarray*}
 &&  2 \sqrt{\pi} \sum_{t_j >0}  u_j(z)\overline{u_j(w_n)} \Re \left( \frac{\G(it_j)}{\G \left(\frac{3}{2} +it_j \right)} \frac{1}{T} \int_{0}^{T}  e^{ i t_j R} dr\right) + O \left( T^{-1} + \sum_{1/2<s_j \leq 1} \frac{1}{T} \int_{0}^{T} \frac{e^{r(1/2-s_j)} }{2s_j-1} d r \right)  \\
&&= O \left( T^{-1} + T^{-1} \sum_{t_j >0}  \frac{u_j(z)\overline{u_j(w_n)}} {|t_j|^{5/2}}     \right)  =  O \left( T^{-1} \right). 
\end{eqnarray*} 
The contribution of Eisenstein series is equal to
\begin{eqnarray*}
\sum_{\frak{a}} \frac{1}{4 \pi} \int_{-\infty}^{\infty}  E_{\mathfrak{a}}(z,1/2+it)   \overline{E_{\mathfrak{a}}(w_n,1/2+it)}    \left( \frac{1}{T} \int_{0}^{T} \frac{h_(t)}{e^{r/2}} d r \right) dt.
\end{eqnarray*}
Let $\phi_{\frak{a},n}(t)$ be as in section 3. Using \cite[p.~87, Lemma~2.4]{phirud} and working as in section \ref{result1compactcase} for  $\phi_{\frak{a},n}(t)$, for every $\epsilon_1>0$ there a $n_0$ such that for every $n \geq n_0$ the contribution of the continuous spectrum is
\begin{eqnarray*}
\sum_{\frak{a}} |E_{\mathfrak{a}}(z,1/2)|^2   + O(\epsilon_1).  
\end{eqnarray*}
Thus, for $\epsilon_1$ sufficiently small and fixed the proposition follows.
\end{proof}

\end{document}